\documentclass[12pt]{amsart}

\usepackage{amssymb}
\usepackage{amssymb,amsmath,amsthm,enumerate,verbatim,bbm}
\usepackage{mathrsfs}

\usepackage{enumitem}

\usepackage{graphicx}

\makeatletter
\@namedef{subjclassname@2010}{%
  \textup{2010} Mathematics Subject Classification}
\makeatother

\usepackage[T1]{fontenc}
\newtheorem{theorem}{Theorem}[section]
\newtheorem{corollary}[theorem]{Corollary}
\newtheorem{lemma}[theorem]{Lemma}
\newtheorem{proposition}[theorem]{Proposition}

\theoremstyle{definition}
\newtheorem{definition}[theorem]{Definition}
\newtheorem{remark}[theorem]{Remark}
\newtheorem{example}[theorem]{Example}

\numberwithin{equation}{section}

\DeclareMathOperator*{\Ran}{\mathrm{Ran}}
\DeclareMathOperator*{\Ker}{\mathrm{Ker}}
\DeclareMathOperator*{\id}{\mathrm{id}}
\DeclareMathOperator*{\jacrad}{\mathrm{rad}}

\DeclareMathOperator*{\codim}{\mathrm{codim}}

\begin{document}

%%%%% To ease editing, for IMPAN journals add:

\baselineskip=17pt

%%%%%%%%%%%%%%%%

\title[Surjective representations of $\mathcal{B}(X)$]{When are full representations of algebras of operators on Banach spaces automatically faithful?}

\author[B. Horv\'ath]{Bence Horv\'ath}
\address{Institute of Mathematics, Czech Academy of Sciences, \v{Z}itn\'a 25, 115 67 Prague 1, Czech Republic}
\email{horvath@math.cas.cz, b.horvath@lancaster.ac.uk}

\date{}

\begin{abstract}
We examine the phenomenon when surjective algebra homomorphisms between algebras of operators on Banach spaces are automatically injective. In the first part of the paper we shall show that for certain Banach spaces $X$ the following property holds: For every non-zero Banach space $Y$ every surjective algebra homomorphism $\psi: \, \mathcal{B}(X) \rightarrow \mathcal{B}(Y)$ is automatically injective. In the second part of the paper we consider the question in the opposite direction: Building on the work of Kania, Koszmider, and Laustsen \textit{(Trans. London Math. Soc., 2014)} we show that for every separable, reflexive Banach space $X$ there is a Banach space $Y_X$ and a surjective but not injective algebra homomorphism $\psi: \, \mathcal{B}(Y_X) \rightarrow \mathcal{B}(X)$.
\end{abstract}



\maketitle

\section{Introduction and preliminaries}	

\subsection{Introduction}

A classical result of Eidelheit (see for example \cite[Theorem~2.5.7]{Dales}) asserts that if $X,Y$ are Banach spaces then they are isomorphic if and only if their algebras of operators $\mathcal{B}(X)$ and $\mathcal{B}(Y)$ are isomorphic as Banach algebras, in the sense that there exists a continuous bijective algebra homomorphism $\psi: \mathcal{B}(X) \rightarrow \mathcal{B}(Y)$. It is natural to ask whether for some class of Banach spaces $X$ this theorem can be strengthened in the following sense: If $Y$ is a non-zero Banach space and $\psi: \mathcal{B}(X) \rightarrow \mathcal{B}(Y)$ is a continuous, surjective algebra homomorphism, is $\psi$ automatically injective? 

It is easy find an example of a Banach space with this property. Indeed, let $X$ be a finite-dimensional Banach space, let $Y$ be a non-zero Banach space and let $\psi: \, \mathcal{B}(X) \rightarrow \mathcal{B}(Y)$ be a surjective algebra homomorphism. Since $\mathcal{B}(X) \simeq M_n(\mathbb{C})$ for some $n \in \mathbb{N}$, simplicity of $M_n(\mathbb{C})$ implies that $\Ker(\psi)= \lbrace 0 \rbrace$. One can also obtain an infinite-dimensional example: If $\mathcal{H}$ be a separable, infinite-dimensional Hilbert space, $Y$ is a non-zero Banach space, and let $\psi: \, \mathcal{B}(\mathcal{H}) \rightarrow \mathcal{B}(Y)$ is a surjective algebra homomorphism, then $\psi$ is automatically injective; see the paragraph before the proof of Theorem \ref{hilbshi}. These simple observations ensure that the following definition is not vacuous.

\begin{definition}
A Banach space $X$ has the \textit{SHAI property} (Surjective Homomorphisms Are Injective) if for every non-zero Banach space $Y$ every surjective algebra homomorphism $\psi: \, \mathcal{B}(X) \rightarrow \mathcal{B}(Y)$ is automatically injective. 
\end{definition}

The purpose of this paper is to initiate the study of this property. The paper is structured as follows.

In the second part of \textit{Section 1} we establish our notations and introduce the necessary background. We begin \textit{Section 2} by giving a list of examples of Banach spaces which lack the SHAI property, see Example \ref{nonex}. We continue by extending our list of examples of Banach spaces with the SHAI property. Since $\ell_2$ possesses this property, it is therefore natural to ask the same question for other classical sequence spaces. We obtain the following result: 

\begin{proposition}\label{ellpshi}
Suppose $X$ is one of the Banach spaces $c_0$ or $\ell_p$ for $1 \leq p \leq \infty$. Then $X$ has the SHAI property.
\end{proposition}

Another way of generalising the $\ell_2$-case is to ask whether all, not necessarily separable Hilbert spaces have the SHAI property. As we will demonstrate, the answer is affirmative:

\begin{theorem}\label{hilbshi}
A Hilbert space of arbitrary density character has the SHAI property.
\end{theorem}

We shall also provide more ``exotic'' examples of Banach spaces with the SHAI property, including Schlumprecht's arbitrarily distortable Banach space $\mathbf{S}$, constructed in \cite{schlumprecht1}:

\begin{theorem}\label{arbdistortshi}
Let $X$ be a complementably minimal Banach space such that it has a complemented subspace isomorphic to $X \oplus X$. Then $X$ has the SHAI property. In particular, Schlumprecht's arbitrarily distortable Banach space $\mathbf{S}$ has the SHAI property.
\end{theorem}

When studying the SHAI property of a Banach space $X$, understanding the complemented subspaces of $X$ and the lattice of closed two-sided ideals of $\mathcal{B}(X)$ appears to be immensely helpful. For the Banach space $X = \textstyle( \bigoplus_{n \in \mathbb{N}} \ell_2^{n} )_Y$, where $Y$ is $c_0$ or $\ell_1$, the complemented subspace structure was studied by Bourgain, Casazza, Lindenstrauss, and Tzafriri in \cite{bclt} and the ideal lattice of $\mathcal{B}(X)$ by Laustsen, Loy, and Read in \cite{llr} and later by Laustsen, Schlumprecht, and Zs\'ak in \cite{ltzs}. Their results allow us to show the following:

\begin{theorem}\label{sumfindimshi}
Let $X:= \textstyle( \bigoplus_{n \in \mathbb{N}} \ell_2^{n} )_Y$, where $Y$ is $c_0$ or $\ell_1$. Then $X$ has the SHAI property.
\end{theorem}

Finally, in \textit{Section 2} we establish a permanence property:

\begin{proposition}\label{directsumshi}
Let $E$ be a Banach space and let $F,G$ be closed subspaces of $E$ with $E = F \oplus G$. If both $F$ and $G$ have the SHAI property then $E$ has the SHAI property.
\end{proposition}

We remark in passing that the stability of the SHAI property under finite sums is of interest to us since $\mathcal{B}(F \oplus G)$ can have a very complicated lattice of closed two-sided ideals even if $\mathcal{B}(F)$ and $\mathcal{B}(G)$ themselves have the simplest possible ideal structure, we refer the reader to \cite{fschzs} and \cite{schzs}. We do not know however if $L_p[0,1]$ possesses the SHAI property for $p \in [1, \infty) \backslash \lbrace 2 \rbrace$.

\textit{Section 3} is devoted entirely to construct Banach spaces which fail the SHAI property in a rather non-trivial manner; for every separable, reflexive Banach space $X$ we find a Banach space $Y_X$ and a surjective but not injective algebra homomorphism $\Theta: \, \mathcal{B}(Y_X) \rightarrow \mathcal{B}(X)$. More precisely, we prove the following:

\begin{theorem}\label{ourmain}
Let $X$ be a non-zero, separable, reflexive Banach space. For every $S \in \mathcal{B}(Y_X)$ there exists a unique $\Theta(S) \in \mathcal{B}(X)$ and there exists a club subset $D \subseteq [0, \omega_1)$ such that for all $\alpha \in D$ and all $\psi \in X^*$:
\begin{align}
S^*(\delta_{\alpha} \otimes \psi) = \delta_{\alpha} \otimes \Theta(S)^* \psi.
\end{align} 
Moreover, the map $\Theta: \, \mathcal{B}(Y_X) \rightarrow \mathcal{B}(X); \, S \mapsto \Theta(S)$ is a non-injective algebra homomorphism of norm one; and there exists an algebra homomorphism $\Lambda: \, \mathcal{B}(X) \rightarrow \mathcal{B}(Y_X)$ of norm one with $\Theta \circ \Lambda = \id_{\mathcal{B}(X)}$. In particular $\Theta$ is surjective.
\end{theorem}

All necessary terminology and notation will be explained in the subsequent sections. 

\subsection{Preliminaries}

Our notations and terminology are standard. The set of natural numbers not including zero will be denoted by $\mathbb{N}$, and $\mathbb{N}_0 := \mathbb{N} \cup \lbrace 0 \rbrace$. The fields of real and complex numbers are denoted by $\mathbb{R}$ and $\mathbb{C}$, respectively. 

\subsubsection{Banach spaces, their algebras of operators and ideals thereof, Banach algebras}
In what follows, all Banach spaces and Banach algebras are assumed to be over the complex scalar field $\mathbb{C}$. Most of our results extend however verbatim to Banach spaces over the real scalar field $\mathbb{R}$. Whenever an argument of ours holds only in the complex case, we emphasize the importance of the choice of the scalar field.

If $X$ is a Banach space then its dual space is $X^*$ and $\langle \cdot \, , \cdot \rangle$ is the duality bracket between $X$ and $X^*$. The symbol $I_X$ is the identity operator on $X$. The symbol $\mathcal{B}(X,Y)$ stands for the Banach space of bounded linear operators between the Banach spaces $X$ and $Y$, we let $\mathcal{B}(X):= \mathcal{B}(X,X)$. For $T \in \mathcal{B}(X,Y)$ its adjoint is $T^* \in \mathcal{B}(Y^*,X^*)$. If $W,Z$ are closed linear subspaces of $X$ and $Y$, respectively, then for a $T \in \mathcal{B}(X,Y)$ we denote the restriction of $T$ to $W$ by $T \vert_W$, clearly $T \vert_W \in \mathcal{B}(W,Y)$. If $\Ran(T) \subseteq Z$ then $T \vert^Z$ denotes $T$ considered as a bounded linear operator between $X$ and $Z$, that is, $T \vert^Z \in \mathcal{B}(X,Z)$.

The \textit{direct sum} of Banach spaces $X$ and $Y$ will be denoted by $X \oplus Y$. Two Banach spaces $X$ and $Y$ are said to be \textit{isomorphic} if there is a linear homeomorphism between $X$ and $Y$, it will be denoted by $X \simeq Y$. If $X \simeq X \oplus X$ we say that \textit{$X$ is isomorphic to its square}. Throughout this paper, whenever two Banach spaces are isometrically isomorphic we shall identify them when it does not cause any confusion. By an \textit{isomorphism of Banach algebras} $A$ and $B$ we understand that there is an algebra homomorphism between $A$ and $B$ which is also a homeomorphism. This will also be denoted by $A \simeq B$.

The symbols $\mathcal{A}(X)$, $\mathcal{K}(X)$, $\mathcal{S}(X)$, $\mathcal{E}(X)$, $\mathcal{W}(X)$ and $\mathcal{X}(X)$ stand for the closed two-sided ideals of operators which are \textit{approximable}, \textit{compact}, \textit{strictly singular}, \textit{inessential}, \textit{weakly compact} and have \textit{separable range}, respectively. We recall that $T \in \mathcal{B}(X)$ is an \textit{inessential operator} (see \cite[page~489]{pi}) if for every $S \in \mathcal{B}(X)$ it follows that $\dim(\Ker(I_X + ST)) < \infty$ and $\codim_X(\Ran(I_X + ST)) < \infty$; this is equivalent to saying that $I_X +ST$ is a Fredholm operator for every $S \in \mathcal{B}(X)$. It is well-known that $\mathcal{A}(X) \subseteq \mathcal{K}(X) \subseteq \mathcal{S}(X) \subseteq \mathcal{E}(X)$ and $\mathcal{K}(X) \subseteq \mathcal{W}(X) \cap \mathcal{X}(X)$ hold, see for example \cite{caradus}. 

A \textit{character} on a unital Banach algebra $A$ is a unit-preserving algebra homomorphism from $A$ to $\mathbb{C}$. Any such character is necessarily of norm at most $1$ and therefore continuous.

\subsubsection{Idempotents, projections}
Let $R$ be a ring. We say that $p \in R$ is an \textit{idempotent} if $p^2 =p$. If $p,q \in R$ are idempotents then we say that they are \textit{mutually orthogonal} and write $p \perp q$ if $pq=0=qp$. For $p,q \in R$ idempotents we write $p \sim q$ if there exist $a,b \in R$ such that $p=ab$ and $q = ba$, in this case we say that $p$ and $q$ are \textit{equivalent}. If $p,q \in R$ are idempotents, then we write $q \leq p$ whenever $pq=q$ and $qp=q$ hold. This is a partial order on the set of idempotents of $R$. We say that an idempotent $p \in R$ is \textit{minimal} if it is minimal in the set of non-zero idempotents of $R$ with respect to this partial order. We write $q < p$ if both $q \leq p$ and $q \neq p$ hold.

In a $C^*$-algebra $A$ an idempotent $p \in A$ is called a \textit{projection} if it is self-adjoint. A projection is minimal if it is minimal in the set of non-zero projections of $A$ with respect to the partial order $\leq$.

\subsubsection{Simple and semisimple algebras}
We say that a unital algebra $A$ is \textit{simple} if the only non-trivial two-sided ideal in $A$ is $A$. If $A$ is a unital algebra, the \textit{Jacobson radical} of $A$, denoted by $\jacrad(A)$, is the intersection of all maximal left ideals in $A$, and it is a two-sided ideal in $A$. If there are no proper left ideals in $A$ we put $\jacrad(A):= A$. A unital algebra is \textit{semisimple} if its Jacobson radical is trivial. For any Banach space $X$, the Banach algebra $\mathcal{B}(X)$ is well-known to be semisimple but it is not simple whenever $X$ is infinite-dimensional, since $\mathcal{A}(X)$ is a proper non-trivial closed two-sided ideal in $\mathcal{B}(X)$.
	
\section{When surjective algebra homomorphisms are automatically injective}

A classical deep result of B. E. Johnson asserts the following.
\begin{theorem}[Johnson]
If $A,B$ are Banach algebras such that $B$ is semisimple, then every surjective algebra homomorphism $\psi: \, A \rightarrow B$ is automatically continuous.
\end{theorem}
For a modern discussion of this result we refer the reader to \cite[Theorem~5.1.5]{Dales}. In what follows we shall use this fundamental result without explicitly mentioning it.

We first observe that there is a large class of Banach spaces which obviously lack the SHAI property.
\begin{lemma}\label{ifcharnotshi}
Let $X$ be an infinite-dimensional Banach space such that $M_n(\mathbb{C})$ is a quotient of $\mathcal{B}(X)$ for some $n \in \mathbb{N}$. Then $X$ does not have the SHAI property.
\end{lemma}

\begin{proof}
Let $\varphi: \, \mathcal{B}(X) \rightarrow M_n(\mathbb{C})$ be a surjective algebra homomorphism. Since $\mathcal{B}(\mathbb{C}^n) \simeq M_n(\mathbb{C})$ we immediately obtain that that there is a surjective algebra homomorphism $\psi: \, \mathcal{B}(X) \rightarrow \mathcal{B}(\mathbb{C}^n)$ which cannot be injective, since $X$ is infinite-dimensional.
\end{proof}

\begin{remark}
For any $n \in \mathbb{N}$ one can easily find an infinite-dimensional Banach space $X$ such that it satisfies the conditions of Lemma \ref{ifcharnotshi}, that is, $M_n(\mathbb{C})$ is a quotient of $\mathcal{B}(X)$. Indeed, let $E$ be an infinite-dimensional Banach space such that $\mathcal{B}(E)$ has a character $\varphi: \, \mathcal{B}(E) \rightarrow \mathbb{C}$. (Examples of such spaces are given below in Example \ref{nonex}.) Let $X:= \textstyle {\bigoplus_{i =1}^n E}$, then there is an isomorphism between the Banach algebras $\mathcal{B}(X)$ and $M_n(\mathcal{B}(E))$, this latter being the Banach algebra of $(n \times n)$-matrices with entries in $\mathcal{B}(E)$. Since every element $A \in \mathcal{B}(E)$ can be written uniquely as $A = \lambda I_E + T$ for some $\lambda \in \mathbb{C}$ and $T \in \Ker(\varphi)$, it is straightforward to check that
\begin{align}
\psi: \; M_n(\mathcal{B}(E)) \rightarrow M_n(\mathbb{C}); \quad (\lambda_{i,j} I_E + T_{i,j})_{i,j =1}^n \mapsto (\lambda_{i,j})_{i,j=1}^n
\end{align}
defines surjective algebra homomorphism. So in particular $M_n(\mathbb{C})$ is a quotient of $\mathcal{B}(X)$.

In fact, something much stronger can be said then the above: It was observed by Kania and Laustsen in \cite[page~1022]{klideal} that every complex, semisimple, finite-dimensional, unital algebra is isomorphic to $\mathcal{B}(X) / \mathcal{K}(X)$ for a suitable Banach space $X$.
\end{remark}

We recall that an infinite-dimensional Banach space $X$ is \textit{indecomposable}, if there are no closed, infinite-dimensional subspaces $Y,Z$ of $X$ such that $X \simeq Y \oplus Z$. A Banach space $X$ is \textit{hereditarily indecomposable} if every closed, infinite-dimensional subspace of $X$ is indecomposable.

The next example collects a variety of examples from the literature where $\mathcal{B}(X)$ is known to have a character, so $X$ does not have the SHAI property by Lemma \ref{ifcharnotshi}. In examples (1)--(3) this character is shown explicitly and in examples (4)--(7) the character is obtained from a commutative quotient on $\mathcal{B}(X)$. It is not intended to be an exhaustive list. 

\begin{example}\label{nonex}
None of the following spaces $X$ have the SHAI property:
\begin{enumerate}
\item $X$ is a \textit{complex} hereditarily indecomposable Banach space, since by \cite[Theorem~18]{GM0} $\mathcal{B}(X)$ has a character whose kernel is $\mathcal{S}(X)$,
\item $X= \mathcal{J}_p$ where $1<p< \infty$ and $\mathcal{J}_p$ is the $p^{th}$ James space, since by \cite[Paragraph~8]{edmit}, $\mathcal{B}(X)$ has a character whose kernel is $\mathcal{W}(X)$, see also \cite[Theorem~4.16]{laustsenmax1},
\item $X=C[0, \omega_1]$, where $\omega_1$ is the first uncountable ordinal, since by \cite[Paragraph~9]{edmit} $\mathcal{B}(X)$ has a character, see also \cite[Proposition~3.1]{lw},
\item $X= C[0, \omega_{\eta}]$, where $\eta$ is a regular cardinal, and $\omega_{\eta}$ is the smallest ordinal of cardinality $\aleph_{\eta}$, since by \cite[Section~4]{ogden} $\mathcal{B}(X)$ has a character,
\item $X=X_{\infty}$, where $X_{\infty}$ is the indecomposable but not hereditarily indecomposable Banach space constructed by Tarbard in \cite[Chapter~4]{Tarbard}, since $\mathcal{B}(X) / \mathcal{K}(X) \simeq \ell_1(\mathbb{N}_0)$, where the right-hand side is endowed with the convolution product,
\item $X=X_K$, where $K$ is a countable compact Hausdorff space and $X_K$ is the Banach space construced by Motakis, Puglisi, and Zisimopoulou in \cite[Theorem~B]{MPZ}, since $\mathcal{B}(X)/ \mathcal{K}(X) \simeq C(K)$,
\item $X=C(K_0)$, where is $K_0$ is the compact Hausdorff connected \\ ``Koszmider'' space constructed by Plebanek in \cite[Theorem~1.3]{plebanek}, since $\mathcal{B}(X) / \mathcal{W}(X) \simeq C(K_0)$, as shown in \cite[Proposition~3.3]{schlackow}, and it also follows from \cite[Theorem~6.5(i)]{dkkkl},
\item $X= \mathcal{G}$, where $\mathcal{G}$ is the Banach space constructed by Gowers in \cite{gowers1}, since $\mathcal{B}(X) / \mathcal{S}(X) \simeq \ell_{\infty} / c_0$, as shown in \cite[Corollary~8.3]{laustsenmax1}.
\end{enumerate}
\end{example}

The purpose of the following lemma is to show for a certain ``nice'' class of Banach spaces, when studying the SHAI property it is enough to restrict our attention to infinite-dimensional spaces $Y$.

\begin{lemma}\label{suffshi}
Let $X$ be a Banach space such that $X$ contains a complemented subspace isomorphic to $X \oplus X$. Then the following are equivalent:
\begin{enumerate}
\item $X$ has the SHAI property,
\item for any $Y$ infinite-dimensional Banach space any surjective algebra homomorphism $\psi: \, \mathcal{B}(X) \rightarrow \mathcal{B}(Y)$ is automatically injective.
\end{enumerate}
\end{lemma}

\begin{proof}
Let $Y$ be a non-zero Banach space and let $\psi: \, \mathcal{B}(X) \rightarrow \mathcal{B}(Y)$ be a surjective algebra homomorphism, we show that $Y$ must be infinite-dimensional. For assume towards a contradiction it is not; then clearly $\mathcal{B}(Y)$ is finite-dimensional, thus by $\mathcal{B}(X) / \Ker(\psi) \simeq \mathcal{B}(Y)$ we have that $\Ker(\psi)$ is finite-codimensional in $\mathcal{B}(X)$. But $X$ has a complemented subspace isomorphic to $X \oplus X$ therefore by successively applying \cite[Propositions~1.9~and~2.3]{ringfinofopalgs} and \cite[Proposition~1.3.34]{Dales} it follows that $\mathcal{B}(X)$ has no proper ideals of finite codimension, a contradiction.
\end{proof}

We recall that if $A,B$ are unital algebras and $\theta: \, A \rightarrow B$ is a surjective algebra homomorphism then $\theta[\jacrad(A)] \subseteq \jacrad(B)$.

\begin{lemma}\label{noninjbigkernel}
Let $X$ be a Banach space, let $B$ be a unital Banach algebra and let $\psi: \, \mathcal{B}(X) \rightarrow B$ be a continuous, surjective, non-injective algebra homomorphism. Then $\psi[\mathcal{E}(X)] \subseteq \jacrad(B)$. In particular, if $B$ is semisimple then $\mathcal{E}(X) \subseteq \Ker(\psi)$.
\end{lemma}

\begin{proof}
Since $\psi$ is not injective $\mathcal{A}(X) \subseteq \Ker(\psi)$ holds and therefore there exists a unique surjective  algebra homomorphism $\theta: \, \mathcal{B}(X) / \mathcal{A}(X) \rightarrow B$ with $\theta \circ \pi = \psi$ and $\Vert \psi \Vert = \Vert \theta \Vert$, where $\pi: \, \mathcal{B}(X) \rightarrow \mathcal{B}(X) / \mathcal{A}(X)$ is the quotient map. Thus $\theta[\jacrad(\mathcal{B}(X) / \mathcal{A}(X))] \subseteq \jacrad(B)$, which by Kleinecke's theorem \cite[Theorem~5.5.9]{caradus} is equivalent to $\theta[\pi[\mathcal{E}(X)]] \subseteq \jacrad(B)$. This is equivalent to $\psi[\mathcal{E}(X)] \subseteq \jacrad(B)$, as required.
\end{proof}

\begin{lemma}\label{squaremax}
Let $X$ be a Banach space such that $\mathcal{E}(X)$ is a maximal ideal in $\mathcal{B}(X)$ and $X$ has a complemented subspace isomorphic to $X \oplus X$. Then $X$ has the SHAI property.
\end{lemma}

\begin{proof}
Let $Y$ be an infinite-dimensional Banach space and let $\psi: \, \mathcal{B}(X) \rightarrow \mathcal{B}(Y)$ be a surjective algebra homomorphism. Assume towards a contradiction that $\psi$ in not injective. Since $\mathcal{B}(Y)$ is semisimple in view of Lemma \ref{noninjbigkernel} it follows that $\mathcal{E}(X) \subseteq \Ker(\psi)$ must hold. Since $\psi$ is surjective, $\Ker(\psi)$ is a proper ideal thus by maximality of $\mathcal{E}(X)$ in $\mathcal{B}(X)$ it follows that $\Ker(\psi) = \mathcal{E}(X)$. Thus $\mathcal{B}(X) / \mathcal{E}(X) \simeq \mathcal{B}(Y)$, where the right-hand side is simple, due to maximality of $\mathcal{E}(X)$ in $\mathcal{B}(X)$, which is a contradiction. Therefore $\psi$ must be injective thus by Lemma \ref{suffshi} the claim is proven. 
\end{proof}

\begin{remark}
We observe that the condition ``$X$ has a complemented subspace isomorphic to $X \oplus X$'' in the previous lemma cannot be dropped in general. Indeed, let $X$ be a hereditarily indecomposable Banach space, then $\mathcal{E}(X) = \mathcal{S}(X)$ is a maximal ideal in $\mathcal{B}(X)$ but by Example \ref{nonex} $(1)$ the space $X$ does not have the SHAI property. 
\end{remark}

\begin{proof}[Proof of Proposition \ref{ellpshi}]
Let $X$ be $c_0$ or $\ell_p$ for $1 \leq p < \infty$. Gohberg, Markus, and Feldman showed in \cite{gmf} that $\mathcal{A}(\ell_p) = \mathcal{K}(\ell_p) = \mathcal{S}(\ell_p) = \mathcal{E}(\ell_p)$ is the only closed, non-trivial, proper, two-sided ideal in $\mathcal{B}(\ell_p)$. In \cite[page~253]{laustsenloy}, Loy and Laustsen deduced that $\mathcal{W}(\ell_{\infty}) = \mathcal{X}(\ell_{\infty}) = \mathcal{S}(\ell_{\infty}) = \mathcal{E}(\ell_{\infty})$ is the unique maximal ideal in $\mathcal{B}(\ell_{\infty})$. Thus in both cases the result follows from Lemma \ref{squaremax}.
\end{proof}

We remark in passing that it was recently shown by W. B. Johnson, G. Pisier and G. Schechtman in \cite[Theorem~4.2]{jpsch} that $\mathcal{B}(\ell_{\infty})$ has continuum many distinct closed two-sided ideals, thus the use of Lemma \ref{squaremax} in the proof of Proposition \ref{ellpshi} is essential.

We recall that a Banach space $X$ is called \textit{complementably minimal} if every closed, infinite-dimensional subspace of $X$ contains a subspace which is complemented in $X$ and isomorphic to $X$.

\begin{proof}[Proof of Theorem \ref{arbdistortshi}]
Since $X$ is complementably minimal, it follows from \cite[Theorem~6.2]{whitley} that $\mathcal{S}(X)$ is the largest proper two-sided ideal in $\mathcal{B}(X)$. In particlar $\mathcal{E}(X) = \mathcal{S}(X)$ is maximal in $\mathcal{B}(X)$, thus Lemma \ref{squaremax} yields the claim.
	
We recall that Schlumprecht's space $\mathbf{S}$ is isomorphic to it is square and it is complementably minimal, as shown, for example, in \cite{schlumprecht2}, thus the first part of the theorem applies.
\end{proof}

In the following we show that for a Hilbert space $H$ of arbitrary density character, the projections lift from any quotient of $\mathcal{B}(H)$. In what follows, if $(X, \mu)$ is a measure space and $f \in L_{\infty}(X,\mu)$ then 
\begin{align}
M_f : \, L_2(X, \mu) \rightarrow L_2(X, \mu); \quad g \mapsto fg
\end{align}
is called the \textit{multiplication operator by $f$} and is clearly a bounded linear operator.

\begin{lemma}\label{hilblem}
Let $H$ be a Hilbert space and let $\mathcal{J}$ be a closed, two-sided ideal in $\mathcal{B}(H)$. For any projection $p \in \mathcal{B}(H)/ \mathcal{J}$ there exists a projection $P \in \mathcal{B}(H)$ such that $p= \pi(P)$, where $\pi: \, \mathcal{B}(H) \rightarrow \mathcal{B}(H)/ \mathcal{J}$ denotes the quotient map.	
\end{lemma}	

\begin{proof}
Let $p \in \mathcal{B}(H)/ \mathcal{J}$ be a projection. There exists a self-adjoint $A \in \mathcal{B}(H)$ such that $p= \pi(A)$. By the spectral theorem for bounded self-adjoint operators \cite[Chapter~IX.,~Theorem~4.6]{conway} there exists a measure space $(X, \mu)$, a $\mu$-almost everywhere bounded, real-valued function $f$ on $X$ and an isometric isomorphism $U: \, H \rightarrow L_2(X, \mu)$ such that $A= U^{-1} \circ M_f \circ U$. Consequently
\begin{align}
\pi(U^{-1} \circ M_f \circ U) &= \pi(A) = p = p^2 = \pi(A^2) = \pi(U^{-1} \circ M_{f^2} \circ U),
\end{align}
which is equivalent to 
\begin{align}\label{inideal1}
U^{-1} \circ M_{f-f^2} \circ U = U^{-1} \circ (M_f - M_{f^2}) \circ U \in \mathcal{J}.
\end{align}
Let $\tilde{f}$ be a representative of the class $f$ and let $h$ be the class of $\mathbf{1}_{[\tilde{f} \geq 1/2]}$, the indicator function of the set $[\tilde{f} \geq 1/2]:= \lbrace x \in X : \, \tilde{f}(x) \geq 1/2 \rbrace$. Clearly $h \in L_{\infty}(X, \mu)$ is well-defined and $P:= U^{-1} \circ M_h \circ U \in \mathcal{B}(H)$ is a projection. We show that $p = \pi(P)$, which is equivalent to showing that $U^{-1} \circ M_{f-h} \circ U \in \mathcal{J}$. We first observe that it is enough to find $g \in L_{\infty}(X, \mu)$ such that $g(f-f^2)=h-f$. Indeed, if such a function $g$ were to exist then $M_g \circ M_{f-f^2} = M_{h-f}$ and consequently
\begin{align}
U^{-1} \circ M_{h-f} \circ U &= U^{-1} \circ M_g \circ M_{f-f^2} \circ U \notag \\
&= (U^{-1} \circ M_g \circ U) \circ (U^{-1} \circ M_{f-f^2} \circ U) \in \mathcal{J}
\end{align}
holds by Equation (\ref{inideal1}) and the fact that $\mathcal{J}$ is an ideal in $\mathcal{B}(H)$.

Thus let $\tilde{g}: \, X \rightarrow \mathbb{R}$ be the following function:
\begin{align}
\tilde{g}(x) := \left\{
\begin{array}{l l}
1/(\tilde{f}(x)-1) & \quad \text{if  } \tilde{f}(x) < 1/2 \\
1/ \tilde{f}(x) & \quad \text{otherwise.} \\
\end{array} \right.
\end{align}
Let $g$ be the class of $\tilde{g}$, clearly $g$ is $\mu$-almost everywhere bounded by $2$. A simple calculation shows that
\begin{align}
\tilde{g}(x)(\tilde{f}(x)-\tilde{f}^2(x)) &= \left\{
\begin{array}{l l}
(\tilde{f}(x)- \tilde{f}^2(x)/(\tilde{f}(x)-1) & \quad \text{if  } \tilde{f}(x) < 1/2 \\
(\tilde{f}(x)-\tilde{f}^2(x)/ \tilde{f}(x) & \quad \text{otherwise,} \\
\end{array} \right.
\end{align}
so $\tilde{g}(x)(\tilde{f}(x)- \tilde{f}^2(x)) = \mathbf{1}_{[\tilde{f} \geq 1/2]}(x) - \tilde{f}(x)$ holds for every $x \in X$. Consequently $g(f-f^2)=h-f$, which proves the claim.
\end{proof}

We recall that in a ring $R$ if $I \trianglelefteq R$ is a two-sided ideal and $p,q \in R$ are idempotents with $p \sim q$ then $p \in I$ if and only if $q \in I$. In a $C^*$-algebra $A$ an idempotent $e \in A$ is a projection if and only if $\Vert e \Vert \leq 1$.

The following lemma is straightforward, we omit its proof. 

\begin{lemma}\label{prelift1} \
\begin{enumerate}
\item Let $X$ be a Banach space and suppose $Q \in \mathcal{B}(X)$ is an idempotent such that $\Ran(Q)$ is isomorphic to its square. Then there exist mutually orthogonal idempotents $Q_1,Q_2 \in \mathcal{B}(X)$ with $Q_1,Q_2 \sim Q$ and $Q_1 +Q_2 = Q$.
\item Let $H$ be a Hilbert space and suppose $Q \in \mathcal{B}(H)$ is a projection with infinite-dimensional range. Then there exist mutually orthogonal projections $Q_1,Q_2 \in \mathcal{B}(H)$ with $Q_1,Q_2 \sim Q$ and $Q_1 +Q_2 = Q$.
\end{enumerate}
\end{lemma}

\begin{corollary}\label{lift1} \
\begin{enumerate}
\item Let $X$ be a Banach space and let $\mathcal{J} \trianglelefteq \mathcal{B}(X)$ be a closed, two-sided ideal. Suppose $Q \in \mathcal{B}(X)$ is an idempotent such that $\Ran(Q)$ is isomorphic to its square and $Q \notin \mathcal{J}$. Then there exist mutually orthogonal idempotents $Q_1,Q_2 \in \mathcal{B}(X)$ with $Q_1,Q_2 \notin \mathcal{J}$ such that $\pi(Q_1), \pi(Q_2) < \pi(Q)$, where $\pi: \, \mathcal{B}(X) \rightarrow \mathcal{B}(X) / \mathcal{J}$ is the quotient map.
\item Let $H$ be a Hilbert space and let $\mathcal{J} \trianglelefteq \mathcal{B}(H)$ be a closed, two-sided ideal. Suppose $Q \in \mathcal{B}(H)$ projection such that $Q \notin \mathcal{J}$. Then there exist mutually orthogonal projections $Q_1,Q_2 \in \mathcal{B}(H)$ with $Q_1,Q_2 \notin \mathcal{J}$ such that $\pi(Q_1), \pi(Q_2) < \pi(Q)$, where $\pi: \, \mathcal{B}(H) \rightarrow \mathcal{B}(H) / \mathcal{J}$ is the quotient map.
\end{enumerate}
\end{corollary}

\begin{proof}
$(1)$ By Lemma \ref{prelift1} $(1)$ there exist mutually orthogonal idempotents $Q_1,Q_2 \in \mathcal{B}(X)$ with $Q_1 + Q_2 = Q$ and $Q_1,Q_2 \sim Q$. For $i \in \lbrace 1,2 \rbrace$ we immediately get $Q_i \leq Q$ and thus $\pi(Q_i) \leq \pi(Q)$. Since $Q_i \sim Q$, the condition $Q \notin \mathcal{J}$ is equivalent to $Q_i \notin \mathcal{J}$. Also, for $i,j \in \lbrace 1,2 \rbrace$ if $i \neq j$ then $Q_j = Q - Q_i$ thus $\pi(Q_i) \neq \pi(Q)$.

$(2)$ Immediate from the first part of this corollary and Lemma \ref{prelift1} $(2)$.
\end{proof}

We recall a folklore lifting result for ``Calkin'' algebras of Banach spaces, this will be essential in the proof of Theorem \ref{sumfindimshi}. A convenient reference for the proof of this lemma is \cite[Lemma~2.6]{boejohn}. It also follows from the much more general result \cite[Theorem C]{analyticlift}.

\begin{lemma}\label{liftcalkin}
Let $X$ be a Banach space and let $p \in \mathcal{B}(X)/ \mathcal{K}(X)$ be an idempotent. Then there exists an idempotent $P \in \mathcal{B}(X)$ with $p = \pi(P)$ where $\pi: \; \mathcal{B}(X) \rightarrow \mathcal{B}(X)/ \mathcal{K}(X)$ is the quotient map.
\end{lemma}

\begin{proposition}\label{hilbnominproj} \
\begin{enumerate}
\item Let $X$ be a Banach space such that every infinite-dimensional complemented subspace of $X$ is isomorphic to its square. \\ Then $\mathcal{B}(X) / \mathcal{K}(X)$ does not have minimal idempotents.
\item Let $H$ be a Hilbert space and let $\mathcal{J} \trianglelefteq \mathcal{B}(H)$ a non-zero, closed, two-sided ideal. Then $\mathcal{B}(H)/ \mathcal{J}$ does not have minimal projections.	
\end{enumerate}
\end{proposition}

\begin{proof}
$(1)$ Let $p \in \mathcal{B}(X) / \mathcal{K}(X)$ be a non-zero idempotent. By Lemma \ref{liftcalkin} there exists an idempotent $P \in \mathcal{B}(X)$ with $p = \pi(P)$, where $\pi: \, \mathcal{B}(X) \rightarrow \mathcal{B}(X) / \mathcal{K}(X)$ is the quotient map. Clearly $P \notin \mathcal{K}(X)$, equivalently $\Ran(P)$ is infinite-dimensional. Thus by the hypothesis it is isomorphic to its square, consequently Corollary \ref{lift1} $(1)$ implies that there exists an idempotent $Q \in \mathcal{B}(X)$ such that $Q \notin I$ and $\pi(Q) < \pi(P)$.

$(2)$ Let $p \in \mathcal{B}(H) / \mathcal{J}$ be a non-zero projection. By Lemma \ref{hilblem} there exists a projection $P \in \mathcal{B}(H)$ with $p = \pi(P)$, where $\pi: \, \mathcal{B}(H) \rightarrow \mathcal{B}(H) / \mathcal{J}$ is the quotient map. Clearly $P \notin \mathcal{J}$ thus by Corollary \ref{lift1} $(2)$ there exists a projection $Q \in \mathcal{B}(H)$ such that $Q \notin \mathcal{J}$ and $\pi(Q) < \pi(P)$.
\end{proof}

We show that Proposition \ref{hilbnominproj} $(2)$ can be strengthened with the aid of the following simple observation. It is certainly well known among experts, however, we could not locate its proof in the literature, thus we include it here for the convenience of the reader.

\begin{lemma}
If a $C^*$-algebra has minimal idempotents then it has minimal projections.
\end{lemma}

\begin{proof}
Let $A$ be a $C^*$-algebra and suppose $e \in A$ is a minimal idempotent. By \cite[Exercise~3.11(i)]{rordam} there exists a projection $p \in A$ with $p \sim e$. Thus there exist $a,b \in A$ such that $ab=p$ and $ba=e$, consequently $ae=pa$ and $bp=eb$. We show that $p \in A$ is a minimal projection. Since $e \neq 0$ it is clear that $p \neq 0$. Let $q \in A$ be a non-zero projection with $q \leq p$, this is, $pq=q$ and $qp=q$. We define $f:=bqa$, and observe that $f \in A$ is a non-zero idempotent. Indeed, $f^2 = bqabqa = bqpqa = bqa =f$ and $f \neq 0$ otherwise $0 = afb = abqab = pqp = q$ which is impossible. Let us observe that $f \leq e$. Indeed, $ef =ebqa = bpqa = bqa =f$ and similarly $fe=f$ holds. Since $e \in A$ is a minimal idempotent it follows that $e=f$ and consequently $aeb=afb$ holds, equivalently $pab = abqab$ equivalently $p =pqp$ which is just $p=q$. This shows that $p \in A$ is a minimal projection.
\end{proof}

\begin{corollary}\label{hilbnoidemp}
Let $H$ be a Hilbert space and let $\mathcal{J} \trianglelefteq \mathcal{B}(H)$ be a non-zero, closed, two-sided ideal. Then $\mathcal{B}(H)/ \mathcal{J}$ does not have minimal idempotents.
\end{corollary}

Before we prove Theorem \ref{hilbshi}, let us remark here that the case where $H$ is separable immediately follows from well-known facts. Indeed, let $Y$ be a non-zero Banach space and let $\psi: \, \mathcal{B}(H) \rightarrow \mathcal{B}(Y)$ be a continuous, surjective algebra homomorphism. Since $\Ker(\psi)$ is a non-trivial, closed, two-sided ideal in $\mathcal{B}(H)$, by the ideal classification result due to Calkin (\cite{calkin}), $\Ker(\psi) = \lbrace 0 \rbrace$ or $\Ker(\psi) = \mathcal{K}(H)$ must hold. In the latter case, $\text{Cal}(H) := \mathcal{B}(H) / \mathcal{K}(H) \simeq \mathcal{B}(Y)$. (We remark in passing that the ideal of compact operators $\mathcal{K}(H)$ coincides with the operator norm-closure of the ideal of finite-rank operators on $H$, since $H$ has a Schauder basis.) Clearly $\text{Cal}(H)$ is simple and infinite-dimensional. If $Y$ is infinite-dimensional, then $\mathcal{B}(Y)$ is not simple, which is impossible; if $Y$ is finite-dimensional then so is $\mathcal{B}(Y)$, a contradiction. Thus $\psi$ must be injective.

\begin{proof}[Proof of Theorem \ref{hilbshi}]
Let $H$ be a Hilbert space. Let $Y$ be a Banach space and assume towards a contradiction that there exists a surjective, non-injective algebra homomorphism $\psi: \, \mathcal{B}(H) \rightarrow \mathcal{B}(Y)$. Then $\Ker(\psi)$ is non-zero and $\mathcal{B}(H) / \Ker(\psi)$ is isomorphic to $\mathcal{B}(Y)$. This is a contradiction since $\mathcal{B}(H) / \Ker(\psi)$ has no minimal idempotents by Corollary \ref{hilbnoidemp}, whereas $\mathcal{B}(Y)$ clearly does.
\end{proof}

\begin{remark}
	In the proof of Theorem \ref{hilbshi} the spectral theorem played a key role, hence the use of \textit{complex} Hilbert spaces was essential. We show now that the theorem remains true for \textit{real} Hilbert spaces. 

In order to to this, we shall need the notion of the \textit{complexification} of real Banach and Hilbert spaces, and real Banach algebras. We refer the interested reader to \cite[Section~13]{Bonsall} and \cite[Chapter~I, Section~3]{Rickart} for the necessary background information. 

Let $H$ be a real Hilbert space of arbitrary density character, we show that $H$ has the SHAI property. Assume towards a contradiction that there is a non-zero real Banach space $Y$ and a surjective, non-injective homomorphism $\varphi: \, \mathcal{B}(H) \rightarrow \mathcal{B}(Y)$ of real Banach algebras. Let $\widehat{\mathcal{B}(H)}$ and $\widehat{\mathcal{B}(Y)}$ denote the complexifications of the real Banach algebras $\mathcal{B}(H)$ and $\mathcal{B}(Y)$, respectively. We define the map
\begin{align}
\psi: \; \widehat{\mathcal{B}(H)} \rightarrow \widehat{\mathcal{B}(Y)}; \quad (T,S) \mapsto (\varphi(T), \varphi(S)),
\end{align}
this is easily seen to be a surjective homomorphism of complex Banach algebras. Since $\varphi$ is not injective, there is a non-zero $S \in \mathcal{B}(H)$ with $\varphi(S) =0$. Thus $\psi(S,S) = (\varphi(S), \varphi(S))=(0,0)$, hence $\psi$ is not injective. However, $\widehat{\mathcal{B}(H)} \simeq \mathcal{B}(\hat{H})$ and $\widehat{\mathcal{B}(Y)} \simeq \mathcal{B}(\hat{Y})$ as complex Banach algebras, thus there is a surjective, non-injective algebra homomorphism  $\theta: \, \mathcal{B}(\hat{H}) \rightarrow \mathcal{B}(\hat{Y})$ of complex Banach algebras. Since $\hat{H}$ is a complex Hibert space this is impossible in view of Theorem \ref{hilbshi}. Therefore $H$ has the SHAI property, as required.
\end{remark}

We recall the following piece of notation: If $\ell_2^n$ denotes the $n$-dimensional Banach space $\mathbb{C}^n$ with the $\ell_2$-norm, then
\begin{align}
\left( \bigoplus\limits_{n \in \mathbb{N}} \ell_2^n \right)_{\ell_1} := \left\lbrace (x_n)_{n \in \mathbb{N}} : \; (\forall n \in \mathbb{N})(x_n \in \ell_2^n), \; \sum\limits_{n \in \mathbb{N}} \Vert x_n \Vert < \infty\right\rbrace
\end{align}
is a Banach space with the norm $\left\Vert (x_n)_{n \in \mathbb{N}} \right\Vert := \sum\limits_{n \in \mathbb{N}} \Vert x_n \Vert$.

Similarly,
\begin{align}
\left( \bigoplus\limits_{n \in \mathbb{N}} \ell_2^n \right)_{c_0} := \left\lbrace (x_n)_{n \in \mathbb{N}} : \; (\forall n \in \mathbb{N})(x_n \in \ell_2^n), \; \lim\limits_{n \rightarrow \infty} \Vert x_n \Vert =0 \right\rbrace
\end{align}
is a Banach space with the norm $\left\Vert (x_n)_{n \in \mathbb{N}} \right\Vert := \sup\limits_{n \in \mathbb{N}} \Vert x_n \Vert$.

\begin{example}\label{shiex}
For the following (non-Hilbertian) Banach spaces $X$ every infinite-dimensional complemented subspace of $X$ is isomorphic to its square therefore by Proposition \ref{hilbnominproj} $(1)$ the Calkin algebra $\mathcal{B}(X) / \mathcal{K}(X)$ does not have minimal idempotents:
\begin{enumerate}
\item $X= c_0(\lambda)$, where $\lambda$ is an infinite cardinal, since by \cite[Proposition~2.8]{acgjm} every infinite-dimensional complemented subspace of $c_0(\lambda)$ is isomorphic to $c_0(\kappa)$ for some infinite cardinal $\kappa \leq \lambda$, and $c_0(\kappa) \simeq c_0(\kappa) \oplus c_0(\kappa)$,
\item $X= \ell_p$  where $p \in [1, \infty) \backslash \lbrace 2 \rbrace$, since by Pe\l{}czy\'{n}ski's theorem (\cite{pelczynski1}) every infinite-dimensional complemented subspace of $\ell_p$ is isomorphic to $\ell_p$ and $\ell_p \simeq \ell_p \oplus \ell_p$,
\item $X= \ell_{\infty}$, since every infinite-dimensional complemented subspace of $\ell_{\infty}$ is isomorphic to $\ell_{\infty}$ by Lindenstrauss' theorem (\cite{lindenstrauss1}) and $\ell_{\infty} \simeq \ell_{\infty} \oplus \ell_{\infty}$,
\item $X= \ell_{\infty}^c(\lambda)$, where $\lambda$ is an infinite cardinal, since by \cite[Theorem~1.4]{johnkaniaschecht} every infinite-dimensional complemented subspace of $\ell_{\infty}^c(\lambda)$ is isomorphic to $\ell_{\infty}$ or $\ell_{\infty}^c(\kappa)$ for some infinite cardinal $\kappa \leq \lambda$, and $\ell_{\infty}^c(\kappa) \simeq \ell_{\infty}^c(\kappa) \oplus \ell_{\infty}^c(\kappa)$,
\item $X= C[0, \omega^{\omega}]$, where $\omega$ is the first infinite ordinal, since by \cite[Theorem~3]{benyamini1} every infinite-dimensional complemented subspace of $C[0, \omega^{\omega}]$ is isomorphic to $c_0$ or $C[0, \omega^{\omega}]$ and $C[0, \omega^{\omega}] \simeq C[0, \omega^{\omega}] \oplus C[0, \omega^{\omega}]$ by \cite[Remark 2.25 and Lemma 2.26]{rosenthal},
\item $X= \textstyle( \bigoplus_{n \in \mathbb{N}} \ell_2^{n} )_Y$ where $Y$ is $c_0$ or $\ell_1$, since by \cite[Corollary~8.4 and Theorem~8.3]{bclt} every infinite-dimensional complemented subspace of $X$ is isomorphic to $Y$ or $X$ and $X \simeq X \oplus X$ by \cite[Corollary~7(i)]{ckl}.
\end{enumerate}
\end{example}

Before we recall two important results of Laustsen--Loy--Read and Laustsen--Schlumprecht--Zs\'ak, let us remind the reader of the following terminology. For Banach spaces $X$ and $Y$ the symbol $\overline{\mathscr{G}}_Y(X)$ denotes the closed, two-sided ideal of operators on $X$ which \textit{factor through $Y$ approximately}, that is, the closed linear span of the set $\lbrace ST : \, S \in \mathcal{B}(Y,X), T \in \mathcal{B}(X,Y) \rbrace$.

\begin{theorem}\cite[Corollary~5.6]{llr}, \cite[Theorem~2.12]{ltzs}\label{idcharsumfin}
Let $X = \textstyle( \bigoplus_{n \in \mathbb{N}} \ell_2^{n} )_Y$ where $Y$ is $c_0$ or $\ell_1$. Then the lattice of closed, two-sided ideals in $\mathcal{B}(X)$ is given by
\begin{align}
\lbrace 0 \rbrace \subsetneq \mathcal{K}(X) \subsetneq \overline{\mathscr{G}}_Y(X) \subsetneq \mathcal{B}(X).
\end{align}
\end{theorem}

\begin{proof}[Proof of Theorem \ref{sumfindimshi}]
Let $Z$ be a Banach space and let $\psi: \, \mathcal{B}(X) \rightarrow \mathcal{B}(Z)$ be a surjective algebra homomorphism. Since $X \simeq X \oplus X$, by Lemma \ref{suffshi} we may suppose that $Z$ is infinite-dimensional. Since $\mathcal{B}(X) / \Ker(\psi) \simeq \mathcal{B}(Z)$, by Theorem \ref{idcharsumfin} it is enough to show that neither $\Ker(\psi) = \mathcal{K}(X)$ nor $\Ker(\psi) = \overline{\mathscr{G}}_Y(X)$ can hold. The case $\Ker(\psi) = \overline{\mathscr{G}}_Y(X)$ is not possible, since $\overline{\mathscr{G}}_Y(X)$ is a maximal two-sided ideal in $\mathcal{B}(X)$ by Theorem \ref{idcharsumfin} and therefore $\mathcal{B}(X) / \overline{\mathscr{G}}_Y(X)$ is simple as a Banach algebra whereas $\mathcal{B}(Z)$ is not, since $Z$ is infinite-dimensional. To see that $\Ker(\psi) = \mathcal{K}(X)$ cannot hold we observe that $\mathcal{B}(X) / \mathcal{K}(X)$ does not have minimal idempotents by Example \ref{shiex} $(6)$ whereas $\mathcal{B}(Z)$ has continuum many. Consequently  $\Ker(\psi) = \lbrace 0 \rbrace$ must hold, thus proving the claim.
\end{proof}

Finally in this section we shall establish some permanence properties of Banach spaces with the SHAI property. We recall a trivial observation:

\begin{remark}\label{idealsquarezero}
If $X$ is an infinite-dimensional Banach space and $J$ is a closed, two-sided ideal of $\mathcal{B}(X)$ such that $A^2 = 0$ for all $A \in J$ then $J= \lbrace 0 \rbrace$. This follows from the fact that $\mathcal{A}(X)$ is the smallest non-trivial, closed, two-sided ideal in $\mathcal{B}(X)$ and $\mathcal{A}(X)$ has an abundance of non-zero rank-one idempotents.
\end{remark}

\begin{proof}[Proof of Proposition \ref{directsumshi}]
Let $P,Q \in \mathcal{B}(E)$ be idempotents with $F= \Ran(P)$ and $G= \Ran(Q)$. Then $P+Q = I_E$ and $PQ = 0= QP$. Now let $X$ be a non-zero Banach space and let $\psi: \, \mathcal{B}(E) \rightarrow \mathcal{B}(X)$ be a surjective algebra homomorphism. Then $Y:= \Ran(\psi(P))$ and $Z:= \Ran(\psi(Q))$ are closed (complemented) subspaces of $X$. Let us fix $T \in \mathcal{B}(F)$, we observe that $\psi(P \vert_F \circ T \circ P \vert^F) \vert_Y \in \mathcal{B}(Y)$ holds. The only thing we need to check is that the range of $\psi(P \vert_F \circ T \circ P \vert^F) \vert_Y$ is contained in $Y$ which is clearly true since $\psi(P) \circ \psi(P \vert_F \circ T \circ P \vert^F) \circ \psi(P) = \psi(P \vert_F \circ T \circ P \vert^F)$. Consequently the map
\begin{align}
\varphi: \; \mathcal{B}(F) \rightarrow \mathcal{B}(Y); \quad T \mapsto \psi(P \vert_F \circ T \circ P \vert^F) \vert_Y
\end{align}
is well-defined. It is immediate to see that $\varphi$ is a linear map. To see that it is multiplicative, it is enough to observe that $P \vert^F \circ P \vert_F = I_F$ thus by multiplicativity of $\psi$, for any $T,S \in \mathcal{B}(F)$ we obtain $\varphi
(T) \circ \varphi(S) = \varphi(T \circ S)$.

We show that $\varphi$ is surjective. To see this we fix an $R \in \mathcal{B}(Y)$. Then $\psi(P) \vert_Y \circ R \circ \psi(P) \vert^Y \in \mathcal{B}(X)$ so by surjectivity of $\psi$ it follows that there exists $A \in \mathcal{B}(E)$ such that $\psi(A) = \psi(P) \vert_Y \circ R \circ \psi(P) \vert^Y$. Consequently $\psi(P \circ A \circ P) = \psi(P) \circ \psi(A) \circ \psi(P) = \psi(P) \vert_Y \circ R \circ \psi(P) \vert^Y$ and thus by the definition of $\varphi$ we obtain 
\begin{align}
\varphi(P \vert^F \circ A \circ P \vert_F) &=  \psi(P \vert_F \circ P \vert^F \circ A \circ P \vert_F \circ P \vert^F) \vert_Y = \psi(P \circ A \circ P) \vert_Y \notag \\
&= \left( \psi(P) \vert_Y \circ R \circ \psi(P) \vert^Y \right) \Big\vert_Y = R.
\end{align}
This proves that $\varphi$ is a surjective algebra homomorphism. Similarly we can show that
\begin{align}\label{thetahom}
\theta: \; \mathcal{B}(G) \rightarrow \mathcal{B}(Z); \quad T \mapsto \psi(Q \vert_G \circ T \circ Q \vert^G) \big\vert_Z
\end{align}
is a well-defined, surjective algebra homomorphism. Assume first that $Y$ and $Z$ are both non-trivial subspaces of $X$. Since both $F$ and $G$ have the SHAI property it follows that $\varphi$ and $\theta$ are injective. Now let $A \in \Ker(\psi)$ be arbitrary. Then $\psi(A)=0$ implies 
\begin{align}
\varphi (P \vert^F \circ A \circ P \vert_F) &= \psi(P \vert_F \circ  P \vert^F \circ A \circ P \vert_F \circ P \vert^F) \big\vert_Y = \psi (P \circ A \circ P) \vert_Y \notag \\
&= \psi(P) \circ \psi(A) \circ  \psi(P) \vert_Y = 0.
\end{align}
Since $\varphi$ is injective it follows that $P \vert^F \circ A \circ P \vert_F =0$. Using the injectivity of $\theta$ a similar argument shows that $Q \vert^G \circ A \circ Q \vert_G = 0$. We recall that $E \simeq F \oplus G$ and thus every $A \in \mathcal{B}(E)$ can be represented as the $(2 \times 2)$-matrix
\begin{align}
\begin{bmatrix}
P \vert^F \circ A \circ P \vert_F & P \vert^F \circ A \circ Q \vert_G \\
Q \vert^G \circ A \circ P \vert_F& Q \vert^G \circ A \circ Q \vert_G \end{bmatrix}.
\end{align}
From the previous we obtain that whenever $A \in \Ker(\psi)$ then $A$ has the off-diagonal matrix form 
\begin{align}\label{offdiag}
A =
\begin{bmatrix}
0& P \vert^F \circ A \circ Q \vert_G \\
Q \vert^G \circ A \circ P \vert_F& 0
\end{bmatrix}.
\end{align}
On the one hand, since $\Ker(\psi)$ is an ideal in $\mathcal{B}(X)$, we obviously have that $A^2 \in \Ker(\psi)$ whenever $A \in \Ker(\psi)$, thus $A^2$ also has the off-diagonal form
\begin{align}
A^2 =
\begin{bmatrix}
0& P \vert^F \circ A^2 \circ Q \vert_G \\
Q \vert^G \circ A^2 \circ P \vert_F& 0
\end{bmatrix}.
\end{align}
On the other hand, the product of two $(2 \times 2)$ off-diagonal matrices is diagonal and therefore by Equation (\ref{offdiag})
\begin{align}
P \vert^F \circ A^2 \circ Q \vert_G &= 0, \notag \\
Q \vert^G \circ A^2 \circ P \vert_F  &= 0
\end{align}
must also hold. Consequently $A^2 = 0$, thus by Remark \ref{idealsquarezero} the equality $\Ker(\psi) = \lbrace 0 \rbrace$ must hold, equivalently, $\psi$ is injective.

Let us observe that both $Y= \lbrace 0 \rbrace$ and $Z = \lbrace 0 \rbrace$ cannot hold. Indeed, if both $\psi(Q)$ and $\psi(P)$ were zero, then we had $0 = \psi(P+Q) = \psi(I_E) = I_X$, contradicting that $X$ is non-zero. Thus without loss of generality we may assume $Y = \lbrace 0 \rbrace$ and $Z \neq \lbrace 0 \rbrace$. Hence $\psi(P) =0$, thus $\psi(Q) = \psi(P)+ \psi(Q) = \psi(P+Q) = \psi(I_E) = I_X$. This is equivalent to $Z = \Ran(\psi(Q)) = X$, and thus $\mathcal{B}(Z) = \mathcal{B}(X)$. Therefore $\theta : \, \mathcal{B}(G) \rightarrow \mathcal{B}(X)$, defined in Equation (\ref{thetahom}) is a surjective algebra homomorphism. Since $G$ has the SHAI property and $X$ is non-zero, it follows that $\theta$ is injective. Let $A \in \mathcal{B}(E)$ be such that $A \in \Ker(\psi)$. Then
\begin{align}
\theta(Q \vert^G \circ A \circ Q \vert_G) &= \psi (Q \vert_G \circ Q \vert^G \circ A \circ Q \vert_G \circ Q \vert^G) \notag \\
&= \psi (Q \circ A \circ Q) = \psi(Q) \circ \psi(A) \circ \psi(Q) =0.
\end{align}
Since $\theta$ is injective, this is equivalent to $Q \vert^G \circ A \circ Q \vert_G =0$ which in turn is equivalent to $Q \circ A \circ Q =0$. We observe that $Q \neq 0$, otherwise $I_X = \psi(Q) = 0$ which contradicts the fact that $X$ is non-zero. Hence we can choose $x \in \Ran(Q)$ and $\xi \in  E^*$ norm one vectors with $\langle x, \xi \rangle =1$. Assume towards a contradiction that $\psi$ is not injective. Then in particular $x \otimes \xi \in \mathcal{F}(E) \subseteq \Ker(\psi)$, consequently $Q \circ (x \otimes \xi) \circ Q =0$. Thus $0 = (Q \circ (x \otimes \xi) \circ Q)x = \langle Qx, \xi \rangle Qx = \langle x, \xi \rangle x =x$, a contradiction. Hence $\psi$ is injective, and therefore we conclude that $E$ has the SHAI property.
\end{proof}

From Proposition \ref{directsumshi} we immediately obtain the following corollary.

\begin{corollary}\label{permshi1}
If $N \in \mathbb{N}$ and $\lbrace E_i \rbrace_{i =1}^N$ is set of Banach spaces with the SHAI property then $\bigoplus_{i=1}^N E_i$ has the SHAI property.
\end{corollary}

%%%%%%%%%%%%%%%%%%%%%%%%%%%%%%%%%%%%%%%%%%%%%%%%%%%%%%%%%%%%%%%%%%%%%%%%%%

\section{Constructing surjective, non-injective homomorphisms from $\mathcal{B}(Y_X)$ to $\mathcal{B}(X)$}
	
\subsection{First remarks}

\subsubsection{Ordinals as topological spaces and spaces of continuous functions thereof}
If $\alpha$ is an ordinal, then $\alpha^+$ denotes its ordinal successor. Equipped with the order topology,  $\alpha$ and $\alpha^+$ are locally compact and compact Hausdorff spaces, respectively. It is well-known that the one-point (or Alexandroff) compactification of $\alpha$ is $\alpha^+$. In line with the general convention, we let $[0, \alpha):= \alpha$ and $[0, \alpha]:= \alpha^+$. We recall that the first uncountable ordinal is denoted by $\omega_1$.

If $K$ is a compact Hausdorff space then $C(K)$ denotes Banach space of complex-valued functions on $K$, with respect to the supremum norm. The Banach space $C[0, \omega_1]$ is called the \textit{Semadeni space}, since he showed in \cite{semadeni} that $C[0, \omega_1]$ is not isomorphic to its square. If $L$ is a locally compact Hausdorff space, and $\tilde{L}:= L \cup \lbrace \infty \rbrace$ is its one-point compactification, then we introduce $C_0(L):= \lbrace g \in C(\tilde{L}) : \, g(\infty) =0 \rbrace$, the Banach space of continuous functions vanishing at infinity, with respect to the supremum norm. In this notation
\begin{align}
C_0[0, \omega_1) = \lbrace g \in C[0, \omega_1] : \, g(\omega_1)= 0 \rbrace. 
\end{align}
For a countable ordinal $\alpha$ let $\mathbf{1}_{[0, \alpha]}$ denote indicator function of the interval $[0, \alpha]$. Since $[0, \alpha]$ is clopen, it follows that $\mathbf{1}_{[0, \alpha]} \in C_0[0, \omega_1)$.
Also, by a theorem of Rudin \cite[Theorem~6]{Rudinatom}, the Banach space $C[0, \omega_1]^*$ is isometrically isomorphic to the Banach space
\begin{align}
\ell_1(\omega_1^+)= \left\lbrace f: \, [0, \omega_1] \rightarrow \mathbb{C} \, : \, \sum\limits_{\alpha \leq \omega_1} \vert f(\alpha) \vert < \infty \right\rbrace.
\end{align}
The following definition is essential for our purposes: 

A subset $D \subseteq [0, \omega_1)$ is called a \textit{club subset} if $D$ is a closed and unbounded subset of $[0, \omega_1)$. 

The following elementary lemma plays a crucial role in the main theorem of this section, it can be found for example in \cite[Lemma~3.4]{jech}.

\begin{lemma}\label{countableinterclub}
A countable intersection of club subsets is a club subset.
\end{lemma}

We recall that for Banach spaces $X$ and $Y$, whenever $u \in E \otimes F$
\begin{align}
\Vert u \Vert_{\epsilon} := \sup \left\lbrace \left\Vert \sum\limits_{i=1}^n \langle x_i, \varphi \rangle  y_i \right\Vert \, : \, u = \sum\limits_{i=1}^n x_i \otimes y_i, \, \varphi \in X^*, \, \Vert \varphi \Vert \leq 1 \right\rbrace
\end{align}
denotes the \textit{injective tensor norm on $X \otimes Y$}. The vector space $X \otimes Y$ endowed with the norm $\Vert \cdot \Vert_{\epsilon}$ is denoted by $X \otimes_{\epsilon} Y$. The completion of $X \otimes_{\epsilon} Y$ with respect to $\Vert \cdot \Vert_{\epsilon}$ is called the \textit{injective tensor product of $X$ and $Y$} and it is denoted by $X \hat{\otimes}_{\epsilon} Y$. It is well-known (see e.g. \cite[Proposition~3.2]{Ryan}) that for Banach spaces $X$, $Y$, $W$, $Z$ if $S \in \mathcal{B}(X,W)$ and $T \in \mathcal{B}(Y,Z)$ then there exists a unique $S \otimes_{\epsilon} T \in \mathcal{B}(X \hat{\otimes}_{\epsilon} Y, W \hat{\otimes}_{\epsilon} Z)$ such that for every $x \in X$, $y \in Y$ the identity $(S \otimes_{\epsilon} T)(x \otimes y) = (Sx) \otimes (Ty)$ holds. Then $\Vert S \otimes_{\epsilon} T \Vert = \Vert S \Vert \Vert T \Vert$.

It follows from \cite[Section~3.2]{Ryan} that for any Banach space $X$ the Banach space $C([0, \omega_1]; X)$ of continuous functions on $[0, \omega_1]$ with values in $X$ is isometrically isomorphic to the Banach space $C[0, \omega_1] \hat{\otimes}_{\epsilon} X$. The isometric isomorphism
\begin{align}
J : \, C[0, \omega_1] \hat{\otimes}_{\epsilon} X \rightarrow C([0, \omega_1]; X)
\end{align}
is given by
\begin{align}
(J(f \otimes x))(\alpha) = f(\alpha)x \quad (f \in C[0, \omega_1], \, x \in X, \, \alpha \in [0, \omega_1]).
\end{align}

\begin{definition}
Let $X$ be a non-zero Banach space. We define
\begin{align}
Y_X:= \lbrace F \in C([0, \omega_1];X) : \; F(\omega_1)=0 \rbrace. 
\end{align}
\end{definition}

Although we shall not need this, we remark in passing that it follows from the Hahn--Banach Separation Theorem that $C_0[0, \omega_1) \hat{\otimes}_{\epsilon} X$ and $Y_X$ are isometrically isomorphic.

\begin{lemma}\label{ycompl}
Let $X$ be a non-zero Banach space. Then $Y_X$ is a complemented subspace of $C([0, \omega_1];X)$.
\end{lemma}

\begin{proof}
For a fixed $x_0 \in X$ let us define the constant function
\begin{align}
c_{x_0} : \, [0, \omega_1] \rightarrow X; \quad \alpha \mapsto x_0,
\end{align}
obviously $c_{x_0} \in C([0, \omega_1]; X)$. Thus we can define the map
\begin{align}
Q: \, C([0, \omega_1];X) \rightarrow C([0, \omega_1];X);  \quad F \mapsto F - c_{F(\omega_1)}.
\end{align}
It is clear that $Q$ is a bounded linear map with $\Vert Q(F) \Vert \leq 2 \Vert F \Vert$. Now we observe that for any $F \in C([0, \omega_1];X)$ we clearly have $Q(F)(\omega_1)=0$, showing that $Q(F) \in Y_X$. Also, for any $F \in Y_X$ and any $\alpha \in [0, \omega_1]$ we have $(Q(F))(\alpha)=F(\alpha)$, consequently $Q$ is an idempotent with $\Ran(Q)= Y_X$ thus proving the claim.
\end{proof}

With the notations of the proof of Lemma \ref{ycompl}, we define
\begin{align}\label{complczero}
P: \; C[0,\omega_1] \rightarrow C[0, \omega_1], \quad g \mapsto g - c_{g(\omega_1)}.
\end{align}
In particular, $\Ran(P) = C_0[0, \omega_1)$.

\begin{remark}\label{idonyx}
Clearly for any $g \in C[0, \omega_1]$, $x \in X$ and $\alpha \in [0, \omega_1]$ we have $(Q(g \otimes x))(\alpha) = (Pg \otimes x)(\alpha)$. From this it follows that $(P \otimes_{\epsilon} I_X)Q(g \otimes x) = Pg \otimes x = Q(g \otimes x)$, thus by linearity and continuity we obtain
\begin{align}\label{identityonxy}
I_{Y_X} = (P \otimes_{\epsilon} I_X) \vert_{Y_X}.
\end{align}
\end{remark}

\begin{lemma}\label{suff}
Let $X$ be a non-zero Banach space and suppose $\mu, \xi \in (Y_X)^*$ satisfy $\langle f \otimes x, \xi \rangle = \langle f \otimes x, \mu \rangle$ for all $f \in C_0[0, \omega_1)$ and $x \in X$. Then $\xi = \mu$. 
\end{lemma}

\begin{proof}
The definition of $P$ and the hypothesis of the lemma ensure that for any $x \in X$ and $g \in C[0, \omega_1]$ the equality $\langle Pg \otimes x, \xi \rangle = \langle Pg \otimes x, \mu \rangle$ holds. By Remark \ref{idonyx} we have $\langle Q(g \otimes x), \xi \rangle = \langle Q(g \otimes x), \mu \rangle$, equivalently, $\langle g \otimes x, (Q \vert^{Y_X})^* \xi \rangle = \langle g \otimes x, (Q \vert^{Y_X})^* \mu \rangle$ and thus by linearity and continuity of $(Q \vert^{Y_X})^* \mu$ and $(Q \vert^{Y_X})^* \xi$ we obtain that for all $u \in C([0, \omega_1];X)$ the identity $\langle u, (Q \vert^{Y_X})^* \xi \rangle = \langle u, (Q \vert^{Y_X})^* \mu \rangle$ holds. Thus for any $u \in C([0, \omega_1];X)$ we have $\langle Qu, \xi \rangle = \langle Qu, \mu \rangle$ consequently by Lemma \ref{ycompl} for all $v \in Y_X$ we have that $\langle v , \xi \rangle = \langle v, \mu \rangle$, proving the claim.
\end{proof}

\begin{remark}\label{nonsep}
Let $X$ be a non-zero Banach space. It is easy to see that $Y_X$ is not separable. Indeed, let $x_0 \in X$ be such that $\Vert x_0 \Vert = 1$ and let us define the map
\begin{align}
\iota: \; C_0[0, \omega_1) \rightarrow Y_X; \quad f \mapsto f \otimes x_0.
\end{align}
This is clearly a linear isometry, thus, since separability passes to subsets %(see \cite[Corollary~1.12.10]{Megginson})
it follows that $Y_X$ cannot be separable.
\end{remark}

In the following, if $\alpha \leq \omega_1$ is an ordinal, then $\delta_{\alpha} \in C[0, \omega_1]^*$ denotes the \textit{Dirac measure} centred at $\alpha$; that is, the bounded linear functional defined by $\delta_{\alpha}(g):=g(\alpha)$ for $g \in C[0, \omega_1]$.

\begin{remark}
Let $X$ be a non-zero Banach space and let $\alpha \in [0, \omega_1]$ and $\psi \in X^*$ be fixed. We can define a map by
\begin{align}
\delta_{\alpha} \otimes \psi : \; C([0, \omega_1];X) \rightarrow \mathbb{C}; \quad u \mapsto \langle u(\alpha), \psi \rangle,
\end{align}
clearly $\delta_{\alpha} \otimes \psi \in C([0, \omega_1];X)^*$.

Let us observe that $C[0, \omega_1]$ has the approximation property. By \cite[Theorem~6]{Rudinatom} we know that $C[0, \omega_1]^*$ is isometrically isomorphic to $\ell_1(\omega_1^+)$, which has the Radon-Nikod\'ym property, consequently by \cite[Theorem~5.33]{Ryan}, the Banach space $(C[0, \omega_1] \hat{\otimes}_{\epsilon} X)^*$ is isometrically isomorphic to $C[0, \omega_1]^* \hat{\otimes}_{\pi} X^*$, the projective tensor product of $C[0, \omega_1]^*$ and $X^*$ (see for example \cite[Section~2.1]{Ryan}). Equivalently, $C([0, \omega_1];X)^*$ is isometrically isomorphic to $\ell_1(\omega_1^+; X^*)$, the Banach space of summable transfinite sequences on $\omega_1^+$ with entries in $X^*$. This justifies the tensor notation in the definition of the functional $\delta_{\alpha} \otimes \psi$.
\end{remark}
	
\subsection{The construction}
	
Our main theorem relies on the following result of Kania, Koszmider, and Laustsen:

\begin{theorem}\cite[Theorem~1.5]{topdich}\label{lkk} 
For every $T \in \mathcal{B}(C_0[0, \omega_1))$ there exists a unique \\ $\varphi(T) \in \mathbb{C}$ such that there exists a club subset $D \subseteq [0, \omega_1)$ such that for all $f \in C[0, \omega_1)$ and $\alpha \in D$:
\begin{align}
(Tf)(\alpha) = \varphi(T)f(\alpha).
\end{align}
Moreover, $\varphi: \, \mathcal{B}(C_0[0, \omega_1)) \rightarrow \mathbb{C}; \; T \mapsto \varphi(T)$ is a character.
\end{theorem}

In \cite{topdich} the character $\varphi: \, \mathcal{B}(C_0[0, \omega_1)) \rightarrow \mathbb{C}$ of the previous theorem is termed the \textit{Alspach--Benyamini character} and its kernel the \textit{Loy--Willis ideal} of $\mathcal{B}(C_0[0, \omega_1))$, and is denoted by $\mathcal{M}_{LW}$. Partial structure of the lattice of closed two-sided ideals of $\mathcal{B}(C_0[0, \omega_1))$ is given in \cite{kl2}, in particular $\mathcal{E}(C_0[0, \omega_1)) = \mathcal{K}(C_0[0, \omega_1)) \subsetneq \mathcal{M}_{LW}$.

\begin{proof}[Proof of Theorem \ref{ourmain}]
Fix $S \in \mathcal{B}(Y_X)$, $x \in X$ and $\psi \in X^*$. For any $f \in C_0[0, \omega_1)$ we can define the map
\begin{align}\label{sectionsofop}
S_x^{\psi}f: \; [0, \omega_1] \rightarrow \mathbb{C}; \quad \alpha \mapsto \langle (S(f \otimes x))(\alpha), \psi \rangle.
\end{align}
It is clear that $S_x^{\psi}f$ is a continuous map, moreover by $S(f \otimes x) \in Y_X$ we also have $(S_x^{\psi}f)(\omega_1)=0$, consequently $S_x^{\psi}f \in C_0[0, \omega_1)$. This allows us to define the map
\begin{align}
S_x^{\psi}: \; C_0[0, \omega_1) \rightarrow C_0[0, \omega_1); \quad f \mapsto S_x^{\psi}f.
\end{align}
It is clear that $S_x^{\psi}$ is a linear map with $\Vert S_x^{\psi} \Vert \leq \Vert S \Vert \Vert x \Vert \Vert \psi \Vert$. Consequently, by Theorem \ref{lkk} there exists a club subset $D_{x, \psi} \subseteq [0, \omega_1)$ such that for all $\alpha \in D_{x, \psi}$ the equality 
\begin{align}
(S_x^{\psi})^* \delta_{\alpha} = \varphi(S_x^{\psi}) \delta_{\alpha}
\end{align}
holds. We also have $\vert \varphi(S_x^{\psi}) \vert \leq \Vert S \Vert \Vert x \Vert \Vert \psi \Vert$, since $\Vert \varphi \Vert = 1$.
This allows us to define the map
\begin{align}
\tilde{\Theta}_S: X \times X^* \rightarrow \mathbb{C}; \quad (x, \psi) \mapsto \varphi(S_x^{\psi}),
\end{align}
and we have for any $x \in X$ and $\psi \in X^*$ that $\vert \tilde{\Theta}_S(x, \psi) \vert \leq \Vert S \Vert \Vert x \Vert \Vert \psi \Vert$. Now we show that $\tilde{\Theta}_S$ is bilinear. We only check that it is linear in the first variable, linearity in the second variable follows by an analogous argument. Let $x,y \in X$, $\psi \in X^*$ and $\lambda \in \mathbb{C}$ be arbitrary. Fix $f \in C_0[0, \omega_1)$ and $\alpha \in [0, \omega_1]$, then using linearity of the tensor product in the second variable, of $S$ and of the functional $\psi$ it follows that 
\begin{align}
(S_{x + \lambda y}^{\psi} f)(\alpha) &= \langle (S(f \otimes (x + \lambda y)))(\alpha), \psi \rangle \notag \\
&= \langle (S(f \otimes x))(\alpha), \psi \rangle + \lambda \langle (S(f \otimes y))(\alpha), \psi \rangle \notag \\
&= (S_x^{\psi} f)(\alpha) + \lambda (S_y^{\psi} f)(\alpha),
\end{align}
proving $S_{x + \lambda y}^{\psi} = S_x + \lambda S_y^{\psi}$. Since $\varphi$ is linear, $\tilde{\Theta}_S(x+ \lambda y, \psi) = \varphi(S_{x + \lambda y}^{\psi}) =  \varphi(S_x^{\psi}) + \lambda \varphi(S_y^{\psi}) = \tilde{\Theta}_S(x, \psi) + \lambda \tilde{\Theta}_S(y, \psi)$ readily follows, proving linearity of $\tilde{\Theta}_S$ in the first variable. Consequently $\tilde{\Theta}_S$ is a bounded bilinear form on $X \times X^*$. If $\kappa_X: \, X \rightarrow X^{**}$ denotes the canonical embedding then by reflexivity of $X$ the map
\begin{align}
\Theta_S: \; X \rightarrow X; \quad x \mapsto \kappa_X^{-1}(\tilde{\Theta}_S(x, \cdot))
\end{align}
defines a bounded linear operator on $X$ with $\Vert \Theta_S \Vert = \Vert \tilde{\Theta}_S \Vert$ and $\langle \Theta_S (x), \psi \rangle = \tilde{\Theta}_S(x, \psi) = \varphi(S_x^{\psi})$ for all $x \in X$, $\psi \in X^*$.
Thus we can define the map
\begin{align}
\Theta: \; \mathcal{B}(Y_X) \rightarrow \mathcal{B}(X); \quad S \mapsto \Theta_S.
\end{align}
Since $X$ is separable and reflexive it follows that $X^*$ is separable too. Let $\mathcal{Q} \subseteq X$ and $\mathcal{R} \subseteq X^*$ be countable dense subsets. Let us fix $S \in \mathcal{B}(Y_X)$, $x \in \mathcal{Q}$ and $\psi \in \mathcal{R}$. As above, there exists a club subset $D^S_{x, \psi} \subseteq [0, \omega_1)$ such that for any $\alpha \in D^S_{x,\psi}$ and any $f \in C_0[0, \omega_1)$: $(S_x^{\psi}f)(\alpha) = \varphi(S_x^{\psi})f(\alpha)$ and hence
\begin{align}\label{randomeq1}
\langle S(f \otimes x), \delta_{\alpha} \otimes \psi \rangle &= \langle (S(f \otimes x))(\alpha), \psi \rangle = (S_x^{\psi}f)(\alpha) \notag \\
&= f(\alpha) \varphi(S_x^{\psi}) = \langle f(\alpha) \Theta(S)x, \psi \rangle \notag \\
&= \langle f \otimes (\Theta(S)x), \delta_{\alpha} \otimes \psi \rangle.
\end{align}
By Lemma \ref{countableinterclub} it follows that 
\begin{align}
D^S:= \bigcap\limits_{(x,\psi) \in \mathcal{Q} \times \mathcal{R}} D^S_{x, \psi}
\end{align}
is a club subset of $[0, \omega_1)$. Consequently for any $\alpha \in D^S$, any $f \in C_0[0, \omega_1)$ and any $x \in \mathcal{Q}$, $\psi \in \mathcal{R}$, Equation (\ref{randomeq1}) holds.
It is clear that for a fixed $S \in \mathcal{B}(Y_X)$, $f \in C_0[0, \omega_1)$ and $\alpha \in D^S$ the maps
\begin{align}
X& \times X^* \rightarrow \mathbb{C}; \quad (x, \psi) \mapsto \langle S(f \otimes x), \delta_{\alpha} \otimes \psi \rangle, \notag \\
X& \times X^* \rightarrow \mathbb{C}; \quad (x, \psi) \mapsto \langle f \otimes (\Theta(S)x), \delta_{\alpha} \otimes \psi \rangle
\end{align}
are continuous functions between metric spaces and thus by density of $\mathcal{Q} \times \mathcal{R}$ in $X \times X^*$, Equation (\ref{randomeq1}) holds everywhere on $X \times X^*$. In other words, for any $S \in \mathcal{B}(Y_X)$ there exists a club subset $D^S \subseteq [0, \omega_1)$ such that for any $\alpha \in D^S$, $f \in C_0[0, \omega_1)$ and $x \in X$, $\psi \in X^*$
\begin{align}\label{almostfinalformula}
\langle f \otimes x, S^*(\delta_{\alpha} \otimes \psi) \rangle = \langle f \otimes x, \delta_{\alpha} \otimes (\Theta(S)^* \psi) \rangle
\end{align}
holds. Therefore by Lemma \ref{suff} we obtain that for all $\alpha \in D^S$ and $\psi \in X^*$:
\begin{align}\label{finalformula}
S^*(\delta_{\alpha} \otimes \psi) = \delta_{\alpha} \otimes (\Theta(S)^* \psi).
\end{align}
We show that for any $S \in \mathcal{B}(Y_X)$ the operator $\Theta(S)$ is determined by this property. Indeed, suppose $\Theta_1(S),\Theta_2(S) \in \mathcal{B}(X)$ are such that there exist club subsets $D^S_1, D^S_2 \subseteq [0, \omega_1)$ such that for $i \in \lbrace 1,2 \rbrace$, all $\alpha \in D^S_i$ and all $\psi \in X^*$
\begin{align}
S^*(\delta_{\alpha} \otimes \psi) = \delta_{\alpha} \otimes (\Theta_i(S)^* \psi).
\end{align}
Let $\alpha \in D^S_1 \cap D^S_2$, $x \in X$ and $\psi \in X^*$ be fixed. Then
\begin{align}
\langle \Theta_1(S)x, \psi \rangle &= \langle \mathbf{1}_{[0, \alpha]} \otimes x, \delta_{\alpha} \otimes (\Theta_1(S)^* \psi) \rangle \notag \\
&= \langle \mathbf{1}_{[0, \alpha]} \otimes x, S^*(\delta_{\alpha} \otimes \psi) \rangle \notag \\
&= \langle \mathbf{1}_{[0, \alpha]} \otimes x, \delta_{\alpha} \otimes (\Theta_2(S)^* \psi) \rangle \notag \\
&= \langle \Theta_2(S)x, \psi \rangle
\end{align}
and thus $\Theta_1(S)= \Theta_2(S)$.
We are now prepared to prove that $\Theta$ is an algebra homomorphism. To see this let $S,T \in \mathcal{B}(Y_X)$ be fixed. Let $D^T, D^S, D^{TS} \subseteq [0, \omega_1)$ be club subsets satisfying Equation (\ref{finalformula}). To see multiplicativity, let $\alpha \in D^T \cap D^S \cap D^{TS}$, $x \in X$ and $\psi \in X^*$ be arbitrary. Then we obtain:
\begin{align}
\delta_{\alpha} \otimes (\Theta(TS)^* \psi) &=
(TS)^* (\delta_{\alpha} \otimes \psi) = S^*T^* (\delta_{\alpha} \otimes \psi) \notag \\
&= S^* (\delta_{\alpha} \otimes (\Theta(T)^* \psi)) \notag \\
&= \delta_{\alpha} \otimes (\Theta(S)^* \Theta(T)^* \psi) \notag \\
&= \delta_{\alpha} \otimes ((\Theta(T) \Theta(S))^* \psi),
\end{align}
hence $\Theta(TS)^* \psi = (\Theta(T) \Theta(S))^* \psi$, so $\Theta(TS)^* = (\Theta(T) \Theta(S))^*$, equivalently \\
$\Theta(TS)= \Theta(T) \Theta(S)$.

Linearity can be shown with analogous reasoning.

For any $S \in \mathcal{B}(Y_X)$ we have $\Vert \Theta(S) \Vert = \Vert \tilde{\Theta}_S \Vert \leq \Vert S \Vert$, thus $\Vert \Theta \Vert \leq 1$. \\
We now show that $\Theta$ has a right inverse. Let $P \in \mathcal{B}(C[0, \omega_1])$ be the idempotent operator as in Equation (\ref{complczero}). Let us fix an $A \in \mathcal{B}(X)$. We observe that $S:= (P \otimes_{\epsilon} A) \vert_{Y_X}$ belongs to $\mathcal{B}(Y_X)$. Indeed, for any $g \in C[0, \omega_1]$ and $x \in X$ the identity $((P \otimes_{\epsilon} A)(g \otimes x))(\omega_1)=(Pg)(\omega_1)Ax=0$ holds plainly because $Pg \in C_0[0, \omega_1)$; thus by linearity and continuity of $P \otimes_{\epsilon} A$ in fact $((P \otimes_{\epsilon} A)u)(\omega_1)=0$ for all $u \in C[0, \omega_1] \hat{\otimes}_{\epsilon} X$. This shows that $S \in \mathcal{B}(Y_X)$ and therefore there exists a club subset $D^S \subseteq [0, \omega_1)$ such that Equation (\ref{finalformula}) is satisfied for all $\alpha \in D^S$ and all $\psi \in X^*$. Fix $\alpha \in D^S$. For any $x \in X$ and $\psi \in X^*$ 
\begin{align}
\langle Ax, \psi \rangle &= \langle \mathbf{1}_{[0, \alpha]} \otimes (Ax), \delta_{\alpha} \otimes \psi \rangle = \langle (P \otimes_{\epsilon} A)(\mathbf{1}_{[0, \alpha]} \otimes x), \delta_{\alpha} \otimes \psi \rangle \notag \\
&= \langle \mathbf{1}_{[0, \alpha]} \otimes x, S^*(\delta_{\alpha} \otimes \psi) \rangle \notag \\
&= \langle \mathbf{1}_{[0, \alpha]} \otimes x, \delta_{\alpha} \otimes (\Theta(S)^* \psi) \rangle \notag \\
&= \langle x, \Theta(S)^* \psi \rangle \notag \\
&= \langle \Theta(S)x, \psi \rangle,
\end{align}
and thus $\Theta(S) = A$. In particular, we obtain $\Theta(I_{Y_X}) = I_X$, with $\Vert \Theta \Vert \leq 1$ this yields $\Vert \Theta \Vert =1$. Also, the above shows that the map
\begin{align}
\Lambda : \; \mathcal{B}(X) \rightarrow \mathcal{B}(Y_X); \quad A \mapsto (P \otimes_{\epsilon} A) \vert_{Y_X}
\end{align}
satisfies $\Theta \circ \Lambda = \id_{\mathcal{B}(X)}$. It is immediate that $\Lambda$ is linear with $\Vert \Lambda \Vert \leq 1$. Also, $\Lambda(I_X) = I_{Y_X}$ holds by Equation (\ref{identityonxy}), consequently $\Vert \Lambda \Vert =1$. The map $\Lambda$ is an algebra homomorphism plainly because $P \in \mathcal{B}(C_0[0, \omega_1))$ is an idempotent, therefore $(P \otimes_{\epsilon} A)(P \otimes_{\epsilon}B) = P \otimes_{\epsilon} (AB)$ for every $A,B \in \mathcal{B}(X)$.

It remains to prove that $\Theta$ is not injective. For assume towards a contradiction it is; then $\mathcal{B}(Y_X)$ and $\mathcal{B}(X)$ are isomorphic as Banach algebras. By Eidelheit's Theorem this is equivalent to saying that $Y_X$ and $X$ are isomorphic as Banach spaces. This is clearly nonsense, since for example, $X$ is separable whereas by Remark \ref{nonsep} the Banach space $Y_X$ is not.
\end{proof}

\begin{remark}\label{lwideal}
With the notations established in the proof of Theorem \ref{ourmain} we clearly have in fact
\begin{align}
\Ker(\Theta)= \lbrace S \in \mathcal{B}(Y_X) \, : \, (\forall x \in X)(\forall \psi \in X^*)(S_x^{\psi} \in \mathcal{M}_{LW}) \rbrace,
\end{align}
where $S_x^{\psi}$ is defined by (\ref{sectionsofop}).
\end{remark}

If $X$ is an infinite-dimensional Banach space then $\Ker(\Theta)$ is of course not maximal in $\mathcal{B}(Y_X)$, however, it is not the smallest possible ideal in $\mathcal{B}(Y_X)$. To see this, we need some preliminary observations.

In the following, let $P\in \mathcal{B}(C[0, \omega_1])$ be as in Equation (\ref{complczero}). If $X$ is a non-zero Banach space, we fix $x_0 \in X$ and $\xi \in X^*$ such that $\Vert x_0 \Vert = \Vert \xi \Vert = \langle x_0, \xi \rangle = 1$ and consider the linear isometry
\begin{align}
\iota: \; C_0[0, \omega_1) \rightarrow Y_X; \quad f \mapsto f \otimes x_0.
\end{align}
We also consider the norm one linear map
\begin{align}
\rho: \; C[0, \omega_1] \hat{\otimes}_{\epsilon} X \rightarrow C[0, \omega_1]
\end{align}
which is unique with the property that for any $g \in C[0, \omega_1]$ and $x \in X$ the identity \\ $\rho(g \otimes x) = \langle x, \xi \rangle g$ holds. With this we obtain the following:

\begin{lemma}\label{admiss}
Let $X$ be a non-zero Banach space. Then
\begin{align}
&\Xi: \; \mathcal{B}(C_0[0, \omega_1)) \rightarrow \mathcal{B}(Y_X); \quad S \mapsto \big( (P \vert_{C_0[0, \omega_1)} \circ S \circ P \vert^{C_0[0, \omega_1)}) \otimes_{\epsilon} I_X \big) \vert_{Y_X} \\
&\Upsilon: \; \mathcal{B}(Y_X) \rightarrow \mathcal{B}(C_0[0, \omega_1)); \quad T \mapsto P \vert^{C_0[0, \omega_1)} \circ \rho \vert_{Y_X} \circ T \circ \iota
\end{align}
define norm one linear maps with $\Upsilon \circ \Xi = \id_{\mathcal{B}(C_0[0, \omega_1))}$. Moreover, $\Xi$ is an algebra homomorphism such that $(\Xi(S))_x^{\psi} = \langle x , \psi \rangle S$ for every $x \in X$ and $\psi \in X^*$.
\end{lemma}

\begin{proof}
It is clear that $\big( (P \vert_{C_0[0, \omega_1)} \circ S \circ P \vert^{C_0[0, \omega_1)}) \otimes_{\epsilon} I_X \big) \vert_{Y_X} \in \mathcal{B}(Y_X)$ holds for any $S \in \mathcal{B}(C_0[0, \omega_1))$, thus $\Xi$ is well-defined. It is easy to see that $\Xi$ is linear with $\Vert \Xi \Vert \leq 1$. From Equation (\ref{identityonxy}) it follows that $\Xi(I_{C_0[0, \omega_1)}) = I_{Y_X}$, thus $\Vert \Xi \Vert =1$. The map $\Xi$ is multiplicative simply by the defining property of injective tensor products of operators. Let $S \in \mathcal{B}(C_0[0, \omega_1))$, $x \in X$ and $\psi \in X^*$ be fixed. Then for any $f \in C_0[0, \omega_1)$ and $\alpha \leq \omega_1$ ordinal
\begin{align}
\big( (\Xi(S))_x^{\psi}f \big)(\alpha) = \langle (\Xi(S)(f \otimes x))(\alpha), \psi \rangle = \langle (Sf)(\alpha)x, \psi \rangle = (Sf)(\alpha) \langle x, \psi \rangle,
\end{align}
thus $(\Xi(S))_x^{\psi} = \langle x , \psi \rangle S$ indeed.

Linearity of $\Upsilon$ is immediate, so is $\Vert \Upsilon \Vert \leq 1$. Since $\Upsilon(I_{Y_X})= I_{C_0[0, \omega_1)}$ follows from the definition of $\Upsilon$, we obtain $\Vert \Upsilon \Vert =1$ as required.

It remains to show that $\Upsilon \circ \Xi = \id_{\mathcal{B}(C_0[0, \omega_1))}$. For any $S \in \mathcal{B}(C_0[0, \omega_1))$  and $f \in C_0[0, \omega_1)$
\begin{align}
\Upsilon(\Xi(S))f &=(P \vert^{C_0[0, \omega_1)} \circ \rho \vert_{Y_X} \circ \Xi(S) \circ \iota)f \notag \\
&= (P \vert^{C_0[0, \omega_1)} \circ \rho \vert_{Y_X} \circ \Xi(S))(f \otimes x_0) \notag \\
&= (P \vert^{C_0[0, \omega_1)} \circ \rho \vert_{Y_X})(Sf \otimes x_0) \notag \\
&= P \vert^{C_0[0, \omega_1)} (\langle x_0, \xi \rangle Sf) \notag \\
&= Sf,
\end{align}
consequently $\Upsilon(\Xi(S)) = S$, which proves the claim.
\end{proof}

\begin{corollary} 
The containment $\mathcal{E}(Y_X) \subsetneq \Ker(\Theta)$ holds.
\end{corollary}

\begin{proof}
By Lemma \ref{noninjbigkernel} it follows that $\mathcal{E}(Y_X) \subseteq \Ker(\Theta)$, we show that the containment is proper. For assume towards a contradiction that $\Ker(\Theta)= \mathcal{E}(Y_X)$. If $S \in \mathcal{M}_{LW}$ then by Lemma \ref{admiss} for all $x \in X$ and $\psi \in X^*$ in fact $(\Xi(S))_x^{\psi} = \langle x, \psi \rangle S \in \mathcal{M}_{LW}$, thus by Remark \ref{lwideal} then $\Xi(S) \in \Ker(\Theta)$ follows. Thus $\Xi(S) \in \mathcal{E}(Y_X)$ by the indirect assumption and since $\mathcal{E}$ is an operator ideal in the sense of Pietsch, it follows from Lemma \ref{admiss} that
\begin{align}
S= \Upsilon(\Xi(S)) = P \vert^{C_0[0, \omega_1)} \circ \rho \vert_{Y_X} \circ \Xi(S) \circ \iota \in \mathcal{E}(C_0[0, \omega_1)).
\end{align}
This yields $\mathcal{M}_{LW} = \mathcal{E}(C_0[0, \omega_1))$, which is a contradiction.
\end{proof}

\subsection*{Acknowledgements}
The majority of the research presented herein was carried out during the author's Ph.D. studies, he is grateful to his supervisors Dr Yemon Choi and Dr Niels J. Laustsen (Lancaster) for their invaluable advice during the preparation of this paper. He is indebted to Dr Saeed Ghasemi (Prague), Dr Gy\"orgy P\'al Geh\'er (Reading), Professor Lajos Moln\'ar (Szeged), Dr Tam\'as Titkos, and Dr Zsigmond Tarcsay (Budapest) for many enlightening conversations. We would like to thank the anonymous referee for his/her many insightful comments, which helped to improve the presentation of the paper a great deal, and for drawing our attention to \cite{schlackow}. The author acknowledges the financial support from the Lancaster University Faculty of Science and Technology and acknowledges with thanks the partial funding received from GA\v{C}R project 19-07129Y; RVO 67985840 (Czech Republic).

%%%%%%%%%%% To ease editing, use normal size for the references:

\end{document}